\documentclass[12pt,oneside]{amsart}

\usepackage{amsthm,amsmath,amssymb}
\usepackage{mathrsfs}
\usepackage{enumerate}
\usepackage{amsfonts}
\usepackage{verbatim}
\usepackage{amsbsy}
\usepackage{amsmath}
\usepackage{amssymb}
\usepackage{mathtools}
\usepackage{MnSymbol}
\usepackage{mathrsfs}
\usepackage{xcolor}

\newtheorem{theorem}{Theorem}[section]
\newtheorem{lemma}[theorem]{Lemma}
\newtheorem{proposition}[theorem]{Proposition}
\newtheorem{conjecture}[theorem]{Conjecture}

\newtheorem*{claim}{Claim}
\newtheorem{question}[theorem]{Question}
 
\theoremstyle{definition}
\newtheorem{definition}[theorem]{Definition}

\theoremstyle{remark}
\newtheorem{remark}[theorem]{Remark}
\newtheorem{example}[theorem]{Example}
\newtheorem{case}{Case}

\newcommand{\qedSymC}{\hfill$\lhd\\$}

\makeatletter
\let\oldfootnote\footnote
\def\footnote{\@ifstar\footnote@star\footnote@nostar}
\def\footnote@star#1{{\let\thefootnote\relax\footnotetext{#1}}}
\def\footnote@nostar{\oldfootnote}
\makeatother

\DeclareMathOperator{\print}{pr}

\begin{document}

\title[Answering Two Conjectures About Parikh Matrices]{Ternary is Still Good for Parikh Matrices}
\author{Robert Merca{\c s}}
\address{ Department of Computer Science\\
	Loughborough University\\  
	United Kingdom}
\email{R.G.Mercas@lboro.ac.uk}
\author{Wen Chean Teh$^*$}
\address{School of Mathematical Sciences\\
		Universiti Sains Malaysia\\
		11800 USM, 
		Malaysia}
\email{dasmenteh@usm.my}


\begin{abstract}
The focus of this work is the study of Parikh matrices with emphasis on two concrete problems. In the first part of our presentation we show that a conjecture by Dick at al.\ in 2021 only stands in the case of ternary alphabets, while providing counterexamples for larger alphabets. In particular, we show that the only type of distinguishability in the case of 3-letter alphabets is the trivial one. The second part of the paper builds on the notion of Parikh matrices for projections of words, discussed initially in this work, and answers, once more in the case of a ternary alphabet, a question posed by Atanasiu et al.\ in 2022 with regards to the minimal Hamming distance in between words sharing a congruency class.
\end{abstract}

\maketitle

\footnote*{$^*$Corresponding author}

\section{Introduction}

{\em Parikh vectors} are by now a classical representation for data. They are in effect arrays of dimension equal to the size of the used alphabet in which every position gives the number of occurrences of the corresponding symbol of the alphabet, in the investigated word. Parikh vectors are computed in linear time, but they are also guaranteed to be logarithmic in the size of the sequence they represent. However they are ambiguous, i.e., multiple words typically share the same Parikh vector. Hence, they cannot be used for lossless data compression. 

Since the Parikh mapping returns, for a given sequence, a vector with the multiplicity of each symbol, later has been proposed a refinement of this mapping which has as result a matrix~\cite{mateescu2000extension}. {\em Parikh matrices} were meant to reduce the ambiguity that the vectors manifest. Therefore, they do not only contain the Parikh vector associated with the word, but also information regarding some of the word's {\em (scattered) subsequences}. In particular, if the symbols of the alphabet were to be concatenated together in lexicographically increasing order as to create \emph{an alphabet word}, the Parikh matrix associated with a given word would provide the multiplicities of all  subsequences of the underlying word which appear as (contiguous) factors in the alphabet word. A subsequence is the result of selecting from a word symbols, in the order of their appearance, without them being necessarily adjacent.

Such a matrix has the same asymptotic complexity as the vector with respect to the size of the considered word, while it is associated to only a significantly smaller number of words. However, it does not normally remove ambiguity entirely, i.e., more than one word can be associated with such a matrix. Two or more words associated with the same Parikh matrix are said to be \emph{amiable}, or \emph{$M$\!-equivalent}, while all words amiable among themselves form an equivalence class. 
Furthermore, while in the case of Parikh vectors the mapping is surjective, this is not the case with Parikh matrices; there exist matrices which are not Parikh, i.e., no word associated to such a matrix can be constructed.

Most of the investigation related to Parikh matrices concerned precisely the ambiguity that the mapping defining these matrices displays. Further refined versions of the concept focused on producing structures that would lead to a reduction of the ambiguity~\cite{alazemi2013several,bera2016algebraic,eugeciouglu2006q}; the methodology encompasses the inclusion of polynomials, various extensions on the mappings, or both. There has also been work on uniquely identifying Parikh matrices. However, due to the difficulty of the problem the majority of the results are either considering binary~\cite{atanasiu2007binary,fosse2004some} or ternary~\cite{atanasiu2014parikh,atanasiu2008parikh, mahalingam2012product,poovanandran2024counting,poovanandran2018elementary, serbanuta2009parikh} alphabets, while for alphabets of higher order the question is rather unexplored. Relative to a reduction of the ambiguity of the Parikh matrices, older work restricted to either considering an alternative order on the alphabet~\cite{atanasiu2002injectivity,poovanandran2019strong} or the additional consideration of the Parikh mapping of the reverse image of the word~\cite{atanasiu2002injectivity}.

An understudied aspect that may reduce the ambiguity of a Parikh matrix concerns the information acquired by altering the word itself, or considering other alterations of the alphabet. With respect to this, the recent work of~\cite{dick2021reducing} presents two different methods that reduce the ambiguity of the original Parikh matrices in the form of $\mathbb{P}$-Parikh matrices and $\mathbb{L}$-Parikh matrices. The former represents a restricted case of Parikh matrices induced by words~\cite{serbanuta2004extending}, where the word is a subsequence of the initial alphabet word. The latter variations of Parikh matrices are the classical ones associated with a specific transformation of the initial word. For both methods, the idea is to use, besides the Parikh matrix associated with the word, another specific matrix as to obtain a reduction in the ambiguity of the original underlying word.

A classical distance, used to quantify how distinguishable sequences are, is the Hamming one. In particular, the Hamming distance in between any two sequences of equal length is given by the number of different symbols, in between the words, at every corresponding position. In~\cite{atanasiu2022erasure}, motivated by a similarity in the investigation of words within an $M$\!-equivalence class to that of the notion of error-correcting codes, the authors make an analysis of Hamming distances in between amiable words. For both binary and ternary alphabets the authors show that there (almost) always exist words that are arbitrarily far away from each other while being $M$\!-equivalent. Furthermore, within each equivalence class of binary words there always exist two that are at Hamming distance $2$ of each other, and no smaller such minimal distance is possible.

\noindent\emph{Main results:} 
In this work we make a more in-depth analysis of the structure that the projections' words of amiable sequences exhibit. To this end, we show that for ternary alphabets amiable words are distinguishable via downwards projections if and only if such projections are trivially deducible, i.e., the considered words contain length two factors that consist of letters not consecutive in the ordered alphabet. Buiding on the notion of distinguishability as an outcome of projections from ternary alphabets, we are also able to show that, within every $M$\!-equivalence class of $3$-letter words at least two words have their Hamming distance as one of the values $\{2,4,7,8\}$.

Our paper is structured as follows. The next section features the preliminaries with respect to the notions and results that we elaborate. Section~\ref{sec:3} answers a conjecture on $\mathbb{P}$-distinguishability by answering it positively only for the ternary alphabet case. Building on the notion of $\mathbb{P}$-distinguishability, in Section~\ref{sec:4} we answer a seemingly unconnected question on the minimal Hamming distance exhibited by a class of $M$\!-equivalent words. We conclude this work with some future observations and suggestions of questions to be approached.

\section{Preliminaries}

Let $\Sigma=\{a_1, a_2, \dotsc,a_s\}$ be a finite nonempty alphabet. The set of all words over $\Sigma$ is denoted by $\Sigma^*$. 
If $\Sigma$ comes with a total ordering $a_1<a_2<\dotsb < a_s$, then we say that $\Sigma$ is an \emph{ordered alphabet} and we write $\Sigma= \{a_1<a_2<\dotsb<a_s\}$. 
In this work, all of the considered alphabets are ordered. 
If $w=a^{p_1}_{i_1}a^{p_2}_{i_2}\dotsm a^{p_n}_{i_n}$, for some 
$a_{i_1}, a_{i_2}, \dotsc, a_{i_n}\in\Sigma^*$ and positive integers $p_1, p_2, \dotsc, p_n$ such that $a_{i_j}\neq a_{i_{j+1}}$, for all $1\le j\le n-1$, then the \textit{print} of $w$, denoted by $\print(w)$, is the word $a_{i_1}a_{i_2}\cdots a_{i_n}$. For $S\subseteq\Sigma$, the \emph{projective morphism} from 
$\Sigma^*$ into $S^*$ is denoted by $\pi_{S}$.
For example, $\pi_{ \{a,c\}} (bacdabc)= acac$, and, for simplicity, we write $\pi_{a,c}( bacdabc)$.

A word $v$ is a \emph{(scattered) subsequence} of $w\in \Sigma^*$ if there exist words $x_1,x_2,\ldots,$ $ x_n, y_0,y_1,\ldots,y_n\in \Sigma^*$ such that $w=y_0x_1y_1\dotsm y_{n-1}x_ny_n$ and $v=x_1x_2\dotsm x_n$.  The number of distinct occurrences of a word $v$ as a subsequence of $w$ is denoted by $|w|_v$. For example, $|aababb|_{ab}=8$ and $|bcacabcba|_{abc}=2$.

\subsection{Parikh Matrices}

For any integer $s\geq 2$, let $\mathcal{M}_s$ denote the multiplicative monoid of $s \times s$ upper triangular matrices with nonnegative integer entries and ones in the principal diagonal.

\begin{definition}[\cite{mateescu2001sharpening}]
	Let $\Sigma=\{a_1<a_2<\cdots <a_s\}$. The \textit{Parikh matrix mapping} with respect to $\Sigma$, denoted $\Psi_\Sigma$, is a morphism from $\Sigma^*$ into $\mathcal{M}_{s+1}$
	defined by:  for every integer $1\le q\le s$, 
	\begin{itemize}
		\item $ [\Psi_\Sigma(a_q) ]_{i, i}=1$, for all $1\le i\le s+1$;
		\item $ [\Psi_\Sigma(a_q) ]_{q, q+1} =1$; and 
		\item all other entries of the matrix $\Psi_\Sigma(a_q)$ are zero; and  
	\end{itemize} 
for every word $w=a_{q_1}a_{q_2}\dotsm a_{q_{\vert w\vert}}\in  \Sigma^*$, we associate the matrix \[\Psi_{\Sigma}(w)= \Psi_{\Sigma}( a_{q_1}) \Psi_{\Sigma}( a_{q_1}) 
\dotsm \Psi_{\Sigma}( a_{q_{\vert w\vert}}).\]
	Matrices of the form $\Psi_\Sigma(w)$, for $w\in\Sigma^*$, are termed as  \textit{Parikh matrices}.
\end{definition}

\begin{theorem}[\cite{mateescu2001sharpening}]\label{1206a}
	Let $\Sigma=\{a_1<a_2< \dotsb<a_s\}$ and  $w\in \Sigma^*$. The matrix $\Psi_{\Sigma}(w)$ has the following properties:
	\begin{itemize}
		\item $[\Psi_\Sigma(w) ]_{i, i}=1$, for each $1\leq i \leq s+1$;
		\item $[\Psi_\Sigma(w) ]_{i, j}=0$, for each $1\leq j<i\leq s+1$;
		\item $[\Psi_\Sigma(w) ]_{i, j+1}=\vert w \vert_{a_ia_{i+1}\dotsm a_j    }$, for each $1\leq i\leq j \leq s$.
	\end{itemize}
\end{theorem}

\begin{example}
	Let $\Sigma=\{a<b<c\}$ and $w=abcabbc$. Then
	\begin{align*}
	\Psi_{\Sigma}(w)&=\Psi_{\Sigma}(a)\Psi_{\Sigma}(b)\Psi_{\Sigma}(c)\Psi_{\Sigma}(a)\Psi_{\Sigma}(b)\Psi_{\Sigma}(b)\Psi_{\Sigma}(c)\\
	&=\begin{psmallmatrix}
	1 & 1 & 0 & 0 \\
	0 & 1 & 0 & 0\\
	0 & 0 & 1 & 0\\
	0 & 0 & 0 & 1
	\end{psmallmatrix}
	\begin{psmallmatrix}
	1 & 0 & 0 & 0 \\
	0 & 1 & 1 & 0\\
	0 & 0 & 1 & 0\\
	0 & 0 & 0 & 1
	\end{psmallmatrix}
	\dotsm
	\begin{psmallmatrix}
	1 & 0 & 0 & 0 \\
	0 & 1 & 0 & 0\\
	0 & 0 & 1 & 1\\
	0 & 0 & 0 & 1
	\end{psmallmatrix}
	=\begin{psmallmatrix}
	1 & 2 & 5 & 6 \\
	0 & 1 & 3 & 4\\
	0 & 0 & 1 & 2\\
	0 & 0 & 0 & 1
	\end{psmallmatrix}
	=\begin{psmallmatrix}
	1 & \vert w\vert_a & \vert w\vert_{ab} & \vert w\vert_{abc} \\
	0 &1 & \vert w\vert_b & \vert w\vert_{bc}\\
	0 & 0 & 1 & \vert w\vert_c\\
	0 & 0 & 0 & 1
	\end{psmallmatrix}.\hfill\lhd
	\end{align*}
\end{example}

\begin{definition}
	Let $\Sigma$ be an ordered alphabet and $w,w'\in \Sigma^*$. We say that $w$ and $w'$ are \mbox{\emph{$M$\!-equivalent}}, denoted $w\equiv_Mw'$, if $\Psi_{\Sigma}(w)=\Psi_{\Sigma}(w')$. If $w\neq w'$, i.e., $w$ is \mbox{$M$\!-equivalent} to a distinct word, then we say that $w$ is \mbox{\emph{$M$\!-ambiguous}};  otherwise, $w$ is \mbox{\emph{$M$\!-unambiguous}}. 
\end{definition}

The $M$\!-equivalence class (or simply $M$\!-class) of $w$ is denoted by $[w]$.

\begin{remark}\label{2912a}
Note that $M$\!-equivalence is both left and right invariant. In other words, if $w\equiv_M w'$, then $uwv\equiv_M uw'v$, for any $u,v\in \Sigma^*$.	
\end{remark}

The following are elementary $M$\!-equivalence preserving rules.

Let $\Sigma=\{a_1<a_2<\cdots<a_s\}$ and $w,w'\in\Sigma^*$.
\begin{itemize}
	\item[$E1$.] If $w=xa_i a_j y$ and $w'=x a_j a_i y$, for some $x,y\in\Sigma^*$ and $|i-j|\ge 2$, then $w\equiv_M w'$.
	\item[$E2$.] If $w=xa_ja_{j+1}ya_{j+1}a_jz$ and $w'=xa_{j+1}a_jya_ja_{j+1}z$, for some $x,y,z\in\Sigma^*$ and $1\leq j\leq s-1$ such that  $|y|_{a_{j-1}} =|y|_{a_{j+2}}=0$, then $w\equiv_M w'$.
\end{itemize}

\begin{theorem}[\cite{atanasiu2007binary, mateescu2004matrix}]\label{1906bb}
Let $\Sigma=\{a<b\}$ and $w, w'\in \Sigma^*$. Then $w\equiv_M w'$ if and only if $w$ can be obtained from $w'$ by finitely many applications of Rule $E_2$. As a consequence,
the following are all $M$\!-unambiguous words over $\Sigma$:
\[\lambda, a^\alpha, b^\alpha, a^\alpha b^\beta, b^\alpha a^\beta,  a^\alpha b a^\beta, b^\alpha a b ^\beta, a^\alpha bab^\beta, b^\alpha ab a^\beta, \quad \alpha, \beta \geq 1.\]
\end{theorem}

The next technical lemma is useful in this work.

\begin{lemma}\label{1011c}
	Let $w\in \{a<b\}^*$ be an $M$\!-ambiguous word with $\vert w\vert_a\geq 1$. Then:
	\begin{enumerate}
		\item $w$ is $M$\!-equivalent to either a word that ends in $a$ or a word of the form $a^{\alpha}b^{\beta}a b^{\gamma}$, for some $\alpha, \beta\geq 0$ and $\gamma \geq 1$.
		\item $w$ is $M$\!-equivalent to either a word that starts with $a$ or a word of the form $b^{\alpha}ab^{\beta} a^{\gamma}$, for some $\alpha\geq 1$ and $\beta, \gamma \geq 0$.		
	\end{enumerate}	  
\end{lemma}

\begin{proof}
	We only prove (1) as the proof of (2) is analogous.	Assume $w$ is not \mbox{$M$\!-equivalent} to any word that ends in $a$. Consider the word $w'$, with $w'\equiv_M w$ and $w'=uab^\gamma$, for some $u\in \{a,b\}^*$ and $\gamma\geq 1$ such that $\gamma$ is minimum. If $u = v_1bav_2$, for some $v_1, v_2\in \{a,b\}^*$, then using Rule $E2$, it follows that
	$w \equiv_M v_1bav_2 ab^\gamma \equiv_M v_1ab v_2 ba b^{\gamma-1}$, contradicting the minimality of $\gamma$. Hence, $ba$ cannot be a factor of $u$ and therefore $u= a^\alpha b^\beta$, for some $\alpha, \beta\geq 0$.	
\end{proof}

Rule $E2$ is sufficient to completely characterize $M$\!-equivalence for the binary alphabet but not for higher alphabets. When restricted to the ternary alphabet, some generalization of Rule $E2$ called Rule $E2.t$ was proposed in~\cite{atanasiu2014parikh}. In this work, a special case, namely Rule $E2.2$, was applied repeatedly in various places. For example, letting  $\Sigma=\{a<b<c\}$ and supposing  $u,v,w\in \Sigma^*$, if $\vert u\vert_c+ \vert v \vert_c= \vert v\vert_a+ \vert w \vert_a$, then the words $abubcvbawcb$ and $baucbvabwbc$ are \mbox{$M$\!-equivalent} according to Rule $E2.2$.  However, to avoid unnecessary technicality, readers who wish to know more about Rule $E2.t$ are referred to~\cite{teh2016conjecture}. Instead, whenever we claim the $M$\!-equivalence of some two words, this can be verified directly by the readers, if necessary, as shown in the following example.

\begin{example}\label{1111a}
Let $w=	ab^{\alpha}c b^{\beta} a b^{\gamma} c b^{\delta} a$, for some $\alpha, \beta, \gamma, \delta\geq 1$.
	If $\alpha = 2k$, for some $k < \delta$, then, letting $ w'= b^kac b^{\beta} a b^{\gamma+\alpha} c b^{\delta-k}a $, we claim that $w\equiv_M w'$. To see this, first it is clear that
	$\vert w\vert_x= \vert w'\vert_x$, for $x\in \{a,b,c\}$.
Also, $\vert w'\vert_{ab}= \beta+ 2(\gamma+\alpha+\delta-k) = \alpha+\beta + 2(\gamma+ \delta) =\vert w\vert_{ab}$. Similarly,
$\vert w'\vert_{bc}= \vert w\vert_{bc}$. Finally,
$\vert w'\vert_{abc}= \beta+ 2(\gamma+\alpha) =  2\alpha +\beta+2\gamma  =\vert w\vert_{abc}$. Therefore, by Theorem~\ref{1206a}, $w\equiv_M w'$.\qedSymC
\end{example}	

We point out that a complete characterization of $M$\!-equivalence for ternary alphabets has been fully described in~\cite{hahn2023equivalence, salomaa2010criteria}. However, for the scope of our current work the previously described Rule $E2.2$ is enough to navigate the words we deal with.
	
Strong $M$\!-equivalence is a variation of $M$\!-equivalence that simultaneously takes into account different orderings of the alphabet \cite{poovanandran2019strong}. As we shall see this $M$\!-equivalence shares some resemblance to $\mathbb{P}$-equivalence to be first coined in Definition~\ref{2812a}.

	\begin{definition}
		Let $\Sigma$ be an alphabet and $w, w'\in \Sigma^*$. We say that $w$ and $w'$ are \emph{strongly $M$\!-equivalent} if $w$ and $w'$ are always $M$\!-equivalent, regardless of the total ordering on $\Sigma$. That is, $\Psi_{\Lambda}(w)= \Psi_{\Lambda}(w')$, for any ordered alphabet $\Lambda$, with $\Sigma$ as the underlying alphabet.
	\end{definition}
	
	The following remark shows why strong $M$\!-equivalence is indeed much stronger than $M$\!-equivalence, and yet natural.
	
	\begin{remark}
		Let $\Sigma=\{a_1<a_2<\cdots<a_s\}$ and $w, w'\in \Sigma^*$. Then, by Theorem~\ref{1206a}, words $w$ and $w'$ are strongly $M$\!-equivalent if and only if, 
		for every $1\leq l\leq s$,
		$\vert w\vert_{a_{i_1}a_{i_2}\dotsm a_{i_l} }= \vert w'\vert_{a_{i_1}a_{i_2}\dotsm a_{i_l} }$, whenever $a_{i_1}, a_{i_2}, \dotsc, a_{i_l}$ are $l$ distinct elements of $\Sigma$.	
	\end{remark}

\subsection{$\mathbb{P}$-Distinguishability}

Let $\Sigma$ be an ordered alphabet and $S\subseteq \Sigma$. We let $\Psi_S$ denote the Parikh matrix mapping with respect to $S$, where it is understood that $S$ is given the total ordering inherited from $\Sigma$.

\begin{definition}[\cite{dick2021reducing}]\label{2511a}
Let $\Sigma$ be an ordered alphabet and $w, w'\in \Sigma^*$. We say that $w$ and $w'$ are \emph{$\mathbb{P}$-distinct} if $\Psi_S(\pi_S(w))\neq \Psi_S( \pi_S (w'))$, for some $S\subseteq \Sigma$. 	
When $w$ is $M$\!-ambiguous, we say that $[w]$ is \emph{$\mathbb{P}$-distinguishable} (or simply distinguishable) if there exists $w'\in [w]$ such that $w$ and $w'$ are $\mathbb{P}$-distinct. Otherwise, $[w]$ is  \emph{$\mathbb{P}$-indistinguishable}.
\end{definition}	

If there exists $w'\in [w]$ such that
$w'$ contains a factor $ab$, for some $a,b\in \Sigma$, such that  $a$ and $b$ are non-consecutive in the ordering of $\Sigma$, then trivially the $M$-class $[w]$ is $\mathbb{P}$-distinguishable and when that happens, we say that $[w]$ is \emph{trivially $\mathbb{P}$-distinguishable} (or simply trivially distinguishable).

\begin{definition}\label{2812a}
Let $\Sigma$ be an ordered alphabet and $w, w'\in \Sigma^*$. We say that $w$ and $w'$ are \emph{$\mathbb{P}$-equivalent} if $\Psi_S(\pi_S(w))= \Psi_S( \pi_S (w'))$, for every $S\subseteq \Sigma$. Equivalently, in this case, $w$ and $w'$ are not $\mathbb{P}$-distinct.	
\end{definition}

\begin{remark}\label{2511b}
	$\mathbb{P}$-equivalence is indeed an equivalence relation. It is a refinement of $M$\!-equivalence as if $w$ and $w'$ are $\mathbb{P}$-equivalent, then they are $M$\!-equivalent but not vice versa. Hence, every $M$\!-equivalence class is a disjoint union of \mbox{$\mathbb{P}$-equivalence} classes. When $w$ is $M$\!-ambiguous, $[w]$ is $\mathbb{P}$-indistinguishable if and only if $[w]$ is itself a $\mathbb{P}$-equivalence class.
\end{remark}

We point out that the concept of $\mathbb{P}$-distinguishability was originally defined in~\cite{dick2021reducing} for Parikh matrices, and there a Parikh matrix $\Psi_{\Sigma}(w)$ is said to be \emph{\mbox{$\mathbb{P}$-distinguishable}} when $w$ is $M$\!-unambiguous or $[w]$ is  $\mathbb{P}$-distinguishable. Since an $M$\!-class  can be identified with its associated Parikh matrix, the difference is not essential.  However,  it causes some disparity if we were to accept $[w]$ to be \mbox{$\mathbb{P}$-distinguishable} when $w$ is $M$\!-unambiguous (see the last line of Remark~\ref{2511b}). Therefore, when it concerns $\mathbb{P}$-distinguishability,  we choose to focus on \mbox{$M$\!-ambiguous} words as in Definition~\ref{2511a}.

\begin{remark}\label{2812b}
	Let $\Sigma=\{a_1<a_2<\cdots<a_s\}$ and $w, w'\in \Sigma^*$. Then, by Definition~\ref{2812a} and Theorem~\ref{1206a}, $w$ and $w'$ are  $\mathbb{P}$-equivalent if and only if, for every $1\leq l\leq s$, $\vert w\vert_{a_{i_1}a_{i_2}\dotsm a_{i_l} }= \vert w'\vert_{a_{i_1}a_{i_2}\dotsm a_{i_l} }$ whenever $1\leq i_1<i_2< \dotsb< i_l\leq s$.	
\end{remark}

Remark~\ref{2812b} shows that, analogous to strong $M$\!-equivalence,  $\mathbb{P}$-equivalence captures a natural combinatorial property between words.

\begin{proposition}[\cite{dick2021reducing}]
Let $\Sigma=\{a_1<a_2<\cdots<a_s\}$ and $w, w'\in \Sigma^*$. If $w=xa_ia_{j}ya_{j}a_iz$ and $w'=xa_{j}a_iya_ia_{j}z$, for some $x,y,z\in\Sigma^*$ and $i<j$ such that $y\in \{ a_i, a_{i+1}, \dotsc, a_j\}^*$, then $w$ and $w'$ are $\mathbb{P}$-equivalent.
\end{proposition}

The above rule is a natural extension of the Rules $E1$ and $E2$ to \mbox{$\mathbb{P}$-equivalence}. Furthermore, some $\mathbb{P}$-equivalence preservation rule that mirrors Rule $E2.2$ was also presented in~\cite{dick2021reducing}. However, in this work, we are mainly concerned with a proposed conjecture about $\mathbb{P}$-distinguishability.

\begin{conjecture}[\cite{dick2021reducing}]\label{160720a}
Let $\Sigma$ be an ordered alphabet and take $w\in\Sigma^*$ to be $M$\!-ambiguous. 
Then the $M$-class $[w]$ is $\mathbb{P}$-distinguishable if and only if $[w]$ is trivially  \mbox{$\mathbb{P}$-distinguishable}. 
\end{conjecture}

\section{$\mathbb{P}$-Distinguishability for Ternary Alphabets}\label{sec:3}

Throughout this section, we fix an ordered alphabet $\{a<b<c\}$ and denote it by $\Sigma_3$. Let $w\in \Sigma_3^*$ be $M$\!-ambiguous. Note that $[w]$ is distinguishable if and only if there exists $w'\in [w]$ such that $\pi_{a,c}(w) \not\equiv_M \pi_{a,c}(w')$ over $\{a<c\}$. 
If $[w]$ contains a word with $ac$ or $ca$ as a factor, then $[w]$ is trivially distinguishable. 

In this section, we prove our first main result, namely that Conjecture~\ref{160720a} holds for the ternary alphabet. Moreover, we also provide some counterexamples to this conjecture for the case of quaternary alphabet. 
To begin, we observe that for any $w\in \Sigma_3^*$, if $[w]$ is not trivially distinguishable, then the letters $a$ and $c$ occur alternately in $\pi_{a,c}(w)$ except, possibly, as prefix or suffix.

\begin{lemma}\label{311219a}
Let $w\in \Sigma_3^*$.  If $ac^na$ or $ca^nc$ is a factor of $\pi_{a,c}(w)$, for some $n\geq 2$, then the $M$-class $[w]$ is trivially $\mathbb{P}$-distinguishable.
\end{lemma}

\begin{proof}
Suppose $ac^na$ is a factor of $\pi_{a,c}(w)$, for some $n\geq 2$, as the other case is similar. It follows that
$ab^\alpha c  u  c b^\beta a$ is a factor of $w$, for some $u\in \{b,c\}^*$ and $\alpha, \beta \geq 0$. Note that $ab^\alpha c  u  c b^\beta a$ is $M$\!-equivalent to $ab^{\alpha-\beta}c b^{\beta}u b^\beta ca$ if $\alpha\geq \beta$,
or $ac b^\alpha u b^\alpha c b^{\beta-\alpha}a$ if $\alpha < \beta$. It implies that 
$w$ is $M$\!-equivalent to a word with $ac$ or $ca$ as a factor, by Remark~\ref{2912a}.
Therefore, $[w]$ is trivially distinguishable.	
\end{proof}

When $w$ is $M$\!-ambiguous and $\vert \print (\pi_{a,c}(w) )\vert \leq 4$, it is possible for $[w]$ to be indistinguishable. For example, consider the word $w= abbabc$.
Our next lemma shows that when $\vert \print (\pi_{a,c}(w) )\vert =3$ and $[w]$ is distinguishable, then $[w]$ is trivially distinguishable.

\begin{lemma}\label{2906c}
	Let $w\in \Sigma_3^*$  be $M$\!-ambiguous with $\vert \print (\pi_{a,c}(w) )\vert =3$. If the $M$-class $[w]$ is $\mathbb{P}$-distinguishable, then $[w]$ is 
	trivially $\mathbb{P}$-distinguishable.
\end{lemma}

\begin{proof}
	Suppose $\print(\pi_{a,c}(w))=aca$ as the other case is analogous. Note that $\vert w\vert_c=1$ by Lemma~\ref{311219a}.
	We argue by contraposition. 
	
	Assume $[w]$ is not trivially distinguishable. Let $w'$ be $M$\!-equivalent to $w$.  Let $w=ucv$ and $w'=u'cv'$, for some $u,v,u',v'\in \{a,b\}^*$. 
	By Lemma~\ref{1011c} and our assumption, $u=a^{\alpha}b^{\beta}a b^{\gamma}$, for some $\alpha, \beta\geq 0$ and $\gamma \geq 1$, and similarly $u'=a^{\alpha'}b^{\beta'}a b^{\gamma'}$, for some $\alpha', \beta'\geq 0$ and $\gamma'\geq 1$. Assume $\alpha \neq \alpha'$ and say $\alpha >\alpha'$. Since $\vert w\vert_{bc}=\vert w'\vert_{bc}$, it follows that $\beta+\gamma=\beta'+\gamma':= L  $. However, $\vert w\vert_{abc} = \alpha(\beta+\gamma) +\gamma> \alpha L \geq \alpha'L+L   \geq  \alpha'(\beta'+\gamma') +\gamma'=  \vert w'\vert_{abc}$, which contradicts the \mbox{$M$\!-equivalence} of $w$ and $w'$. Hence, $\alpha= \alpha'$ and this implies that $\pi_{a,c}(w)=\pi_{a,c} (w')$. Since $w'$ is arbitrary, it follows that $[w]$ is indistinguishable.
\end{proof}

On the other hand, when $w$ is $M$\!-ambiguous and $\vert \print (\pi_{a,c}(w) )\vert \geq 5$, we show that it is impossible for $[w]$ to be indistinguishable. 

\begin{lemma}\label{1011a}
Let $w\in \Sigma_3^*$ with $\vert \print (\pi_{a,c}(w) )\vert \geq 7$. Then the \mbox{$M$-class} $[w]$ is trivially $\mathbb{P}$-distinguishable.
\end{lemma}

\begin{proof}	
	Assume $[w]$ is not trivially distinguishable. By Lemma~\ref{311219a}, $\pi_{a,c}(w)$ must contain  $acacaca$ or $cacacac$ as a factor. We suppose $\pi_{a,c}(w)$ contains the former as a factor as the other case is similar. Hence, by our assumption, $w$ must have a factor of the form 
	$z=ab^{\alpha}c b^{\beta} a b^{\gamma} c b^{\delta} a b^{\nu} c b^{\zeta} a$ for some $\alpha, \beta, \gamma, \delta, \nu, \zeta\geq 1$.
	Note that here $z\equiv_M bab^{\alpha-1}c b^{\beta-1} a b^{\gamma} c b^{\delta+1} a b^{\nu+1} c b^{\zeta-1} a$. In fact, for every $k\leq \min\{\alpha, \beta, \zeta\}$, it can be verified that
	$z\equiv_M b^kab^{\alpha-k}c b^{\beta-k} a b^{\gamma} c b^{\delta+k} a b^{\nu+k} c b^{\zeta-k} a$. In particular, considering $k=\min \{\alpha, \beta, \zeta\}$, by Remark~\ref{2912a}, it follows that $w$ is $M$\!-equivalent to a word with $ac$ or $ca$ as a factor. Therefore, $[w]$ is trivially distinguishable, contradicting our assumption.
\end{proof}

The following lemma is needed to deal with the case $\vert \print (\pi_{a,c}(w) )\vert \in \{5,6\}$.

\begin{lemma}\label{0507b}
Consider a word $w\in\{ab^{\alpha}c b^{\beta} a b^{\gamma} c b^{\delta} a, cb^{\alpha}a b^{\beta} c b^{\gamma} a b^{\delta} c\}$  for some $\alpha, \beta, \gamma, \delta\geq 1$. Suppose $[w]$ is not trivially $\mathbb{P}$-distinguishable. Then $\alpha=\delta=1$.	
\end{lemma}

\begin{proof}
Let $w= ab^{\alpha}c b^{\beta} a b^{\gamma} c b^{\delta} a$, as the other case is similar. 
If $\alpha = 2k$, for some $k\geq 1$, then $w\equiv_M b^kac b^{\beta} a b^{\gamma+\alpha} c b^{\delta-k}a $ if $k<\delta$, or $w\equiv_M b^\delta a b^{\alpha-2\delta}  c b^{\beta} a b^{\gamma+2\delta} ca$ if $k\geq \delta$, contradicting that $[w]$ is not trivially distinguishable.
Hence, $\alpha$ is odd and, symmetrically, $\delta$ is odd as well. 
If $\delta \geq 3$, then $w\equiv_M ab^{\alpha-1}c b^{\beta+2} a b^{\gamma} c b^{\delta-2}ab $. Note that $\alpha-1$ is even and nonzero or else $[w]$ is trivially distinguishable. However, we can apply the argument above on $ab^{\alpha-1}c b^{\beta+2} a b^{\gamma} c b^{\delta-2}a$ to reach a contradiction. Hence, $\delta =1$ and, symmetrically, $\alpha=1$ as well. 		
\end{proof}

A complete list of $M$\!-unambiguous words over $\Sigma_3$ already exists.

\begin{remark}\label{rem:M-unambiguity}
A complete list of $M$\!-unambiguous words over $\Sigma_3$ can be found in the appendix of~\cite{serbanuta2006injectivity}, while its complementary finer list of $M$\!-ambiguous words is available from~\cite{hahn2023equivalence}.
\end{remark}

The following list contains every $M$\!-unambiguous word $w$ over $\Sigma_3$ such that $\vert \print (\pi_{a,c}(w) )\vert \in \{5,6\}$:
\begin{gather*}
	a^m bcb^nab^pcba^q, a^mbcbab^ncbab^p, a^mbcbab^ncbabc^p,\qquad\qquad\\
c^m bab^ncb^pabc^q, c^mbabcb^nabcb^p, c^mbabcb^nabcba^p,\qquad\qquad\\
b^mabcb^nabcba^p, b^mcbab^ncbabc^p,\qquad\qquad\qquad\qquad\\
\qquad \qquad b^mabcbabcbab^n, b^mcbabcbabcb^n \qquad \qquad\text{ with } m,n,p,q\geq 1.	
\end{gather*}

The proofs of the next two lemmas show us that,
when  $\vert \print (\pi_{a,c}(w) )\vert \geq 5$ and $[w]$ is not trivially distinguishable, this list includes every possible $w$.

\begin{lemma}\label{2906b}
Let $w\in \Sigma_3^*$ with $\vert \print (\pi_{a,c}(w) )\vert = 5$. Then either the \mbox{$M$\!-class}  $[w]$ is 
	trivially $\mathbb{P}$-distinguishable or $w$ is $M$\!-unambiguous.
\end{lemma}

\begin{proof}	
Suppose $[w]$ is not trivially distinguishable. We need to show that $w$ is $M$\!-unambiguous. First, suppose $\print (\pi_{a,c}(w) )=acaca$. By Lemma~\ref{311219a}, $\pi_{a,c}(w)= a^m caca^n$, for some $m,n\geq 1$. Since $[w]$ is not trivially distinguishable, by Lemma~\ref{1011c}, it follows that $w$ is $M$\!-equivalent to the following word: 
\[a^{m-1}b^{\alpha}a b^{\beta} c b^{\gamma} a b^{\delta} c b^{\nu} a b^{\zeta} a^{n-1},\]
for some $\beta, \gamma, \delta, \nu \geq 1$ and $\alpha, \zeta \geq 0$.

By Lemma~\ref{0507b} and considering $a b^{\beta} c b^{\gamma} a b^{\delta} c b^{\nu} a$, it follows that $\beta=\nu =1$ or else $[w]$ is trivially distinguishable. Hence, we may now assume that 
\[w=a^{m-1}b^{\alpha}a b c b^{\gamma} a b^{\delta} c b a b^{\zeta} a^{n-1}.\]

\begin{claim}
If $\alpha	\geq 1$, then $m=\delta=1$, and if $\zeta\geq 1$, then $n=\gamma=1$.	
\end{claim}

Suppose $\alpha\geq 1$. If $m\geq 2$, then  $w\equiv_M  
a^{m-2}bab^{\alpha-1}ac b^{\gamma} a b^{\delta+2} c a b^{\zeta} a^{n-1}$,
contradicting that $[w]$ is not trivially distinguishable.
Hence, $w=b^{\alpha}a b c b^{\gamma} a b^{\delta} cb a b^{\zeta} a^{n-1}$. Assume $\delta \geq 2$. Then  $w$ is $M$\!-equivalent to 
the word  $b^{\alpha-1}a b^3 c b^{\gamma} a b^{\delta-2} c b^2 a b^{\zeta} a^{n-1} $. Either $\delta-2=0$ or we apply Lemma~\ref{0507b} on $a b^3 c b^{\gamma} a b^{\delta-2} c b^2 a$. In either case, $[w]$ is trivially distinguishable, which contradicts our assumption. Therefore, $\delta =1$. Similarly, we can prove the second half of our claim.\qedSymC

Now, using the claim and proceeding in cases,  it follows that
\[w=\begin{cases}
a^{m} b c b^{\gamma} a b^{\delta} c b a^n &\text{ if } \alpha=\zeta=0,\\
b^{\alpha} a b c b^{\gamma} a b c b a^{n} &\text{ if } \alpha\geq 1 \text{ and } \zeta=0,\\
a^m b c ba b^{\delta} c b a b^{\zeta} &\text{ if } \alpha =0 \text{ and } \zeta\geq 1,\\
b^{\alpha}a b c b a b c b a b^{\zeta} &\text{ if } \alpha\geq 1 \text{ and } \zeta \geq 1.
\end{cases}
\]
In all cases, $w$ is $M$\!-unambiguous as required.
	
Analogously, when $\print (\pi_{a,c}(w) )= c acac$, it results in another four sets of \mbox{$M$\!-unambiguous} words.
\end{proof}

\begin{lemma}\label{1011b}
Let $w\in \Sigma_3^*$ with $\vert \print (\pi_{a,c}(w) )\vert = 6$. Then either the \mbox{$M$\!-class} $[w]$ is 
	trivially $\mathbb{P}$-distinguishable or $w$ is $M$\!-unambiguous.
\end{lemma}

\begin{proof}
We assume that $\print (\pi_{a,c}(w) )= ac acac$ as the other case is analogous.
Suppose $[w]$ is not trivially distinguishable. 
By Lemma~\ref{311219a}, $\pi_{a,c}(w)= a^m cacac^n$, for some $m,n\geq 1$.  By Lemma~\ref{1011c},
since $[w]$ is not trivially distinguishable, it implies that $w$ is $M$\!-equivalent to the following word: 
\[a^{m-1}b^{\alpha}a b^{\beta} c b^{\gamma} a b^{\delta} c b^{\nu} a b^{\zeta} c b^{\eta} c^{n-1},\]
for some $\beta, \gamma, \delta, \nu, \zeta \geq 1$ and $\alpha, \eta \geq 0$. 

By Lemma~\ref{0507b} and considering  $a b^{\beta} c b^{\gamma} a b^{\delta} c b^{\nu} a$, it follows that $\beta=\nu=1$ or else $[w]$ is trivially distinguishable. Similarly, by considering $ c b^{\gamma} a b^{\delta} c b^{\nu} a b^{\zeta} c$, it implies that 
$\gamma=\zeta =1$.
Hence, we may assume that $w=a^{m-1}b^{\alpha}abcbab^{\delta}cbabcb^{\eta}c^{n-1}$. Note that  $babcbab^{\delta}cbabc \equiv_M abbcbbab^{\delta-1} bcacb$ and thus $\alpha =0$ since, otherwise, $[w]$ is trivially distinguishable. Similarly, we can argue that $\eta=0$. Therefore, $w=a^{m} b c b a b^{\delta} c b a b c^{n} $, which is $M$\!-unambiguous.
\end{proof}

We are now ready to prove Conjecture~\ref{160720a} for the  ternary alphabet.

\begin{theorem}\label{0507a}
Let $w\in \Sigma_3^*$ be $M$\!-ambiguous. If $[w]$ is \mbox{$\mathbb{P}$-distinguishable}, then $[w]$ is trivially $\mathbb{P}$-distinguishable.	
\end{theorem}

\begin{proof}
Clearly, $\vert\print (\pi_{a,c}(w))\vert\geq 1$ or else $w\in \{b\}^*$ and  hence $w$ is \mbox{$M$\!-unambiguous}. Suppose $[w]$ is distinguishable. Observe that here
$\vert\print (\pi_{a,c}(w))\vert\neq 1$, since, otherwise, $w \in\{a,b\}^* \cup \{b,c\}^*$ and thus $[w]$ is indistinguishable.
By Lemmas~\ref{2906c},~\ref{1011a},~\ref{2906b}, and~\ref{1011b}, it remains to deal with when $\vert \print (\pi_{a,c}(v) )\vert \in \{2,4\}$,  for every $ v\in[w]$.

Assume $[w]$ is not trivially distinguishable. First, we deal with the situation where
$\vert\print (\pi_{a,c}(w) )\vert =4$. We may suppose $\pi_{a,c}(w)=a^mcac^n$, for some positive integers $m$ and $n$, as the other case is analogous. Since $[w]$ is not trivially distinguishable, by Lemma~\ref{1011c} and without loss of generality, we may assume 
\[w=a^{m-1} b^{\alpha} a b^{\beta} c b^{\gamma} a b^{\delta} c b^{\nu} c^{n-1},\]
for some $\beta, \gamma, \delta \geq 1$ and $\alpha, \nu \geq 0$. 
Since $[w]$ is distinguishable, there is some $w'\in [w]$ such that $\pi_{a,c}(w')\not\equiv_M \pi_{a,c}(w)$ over $\{a<c\}$. 
Note that $\print(\pi_{a,c}(w'))$ cannot be $acac$ as that forces $\pi_{a,c}(w')= a^mcac^n=\pi_{a,c}(w)$ by Lemma~\ref{311219a}.
It follows that $\print(\pi_{a,c}(w'))$ is  $caca$, $ac$, or $ca$ and thus we have the following cases. 

\setcounter{case}{0}

\begin{case}
	$\pi_{a,c}(w')= c^{n+1}a^{m+1}$.	
\end{case}

This case is impossible since $\vert w'\vert_{abc}=0$, but $\vert w\vert_{abc}>0$.\qedSymC
\begin{case}
	$\pi_{a,c}(w')= a^{m+1}c^{n+1}$.
\end{case}

By Lemma~\ref{1011c}, 
$w'$ is $M$\!-equivalent to the word $a^{m} b^{\alpha'} a b^{\beta'} c b^{\gamma'} c^{n}$, for some $\beta'\geq 1$ and $\alpha',\gamma'\geq 0$. Therefore,
\[a^{m-1} b^\alpha a b^\beta c b^\gamma a b^\delta c b^\nu c^{n-1}\equiv_M a^{m} b^{\alpha'} a b^{\beta'} c b^{\gamma'} c^{n}.\] 
Due to Remark~\ref{2912a}, $M$\!-equivalence is left and right invariant, and it follows that 
\[\underbrace{b^\alpha a b^\beta c b^\gamma a b^\delta c b^\nu}_{z} \equiv_M  \underbrace{a b^{\alpha'} a b^{\beta'} c b^{\gamma'} c}_{z'}.\] 
Note that every letter $b$ in $z$ contributes at most three to the sum $\vert z \vert_{ab} + \vert z \vert_{bc}$,
and clearly $\vert z \vert_{ab} + \vert z \vert_{bc}< 3\vert z\vert_{b}$ as $\gamma\geq 1$. Similarly, it can be verified that $\vert z'\vert_{ab}+ \vert z'\vert_{bc}> 3\vert z'\vert_b$. It follows that
$\vert z \vert_{ab} + \vert z \vert_{bc}< \vert z'\vert_{ab}+ \vert z'\vert_{bc}$,
contradicting the fact that $z\equiv_M z'$. \qedSymC

\begin{case}
	$\pi_{a,c}(w')= c^{n}aca^{m}$.		
\end{case}

Similarly, by Lemma~\ref{1011c} and without loss of generality, we may assume  
\[w'=c^{n-1} b^{\alpha'} c b^{\beta'}a b^{\gamma'}c b^{\delta'} a b^{\nu'}a^{m-1},\]
for some $\beta', \gamma', \delta'\geq 1$ and $\alpha',\nu'\geq 0$.
Note that every letter $b$ in $w$ contributes either $m+n$ or $m+n+1$ 
 to the sum $\vert w \vert_{ab} + \vert w \vert_{bc}$
 and since $\beta\geq 1$, it follows that
 $\vert w \vert_{ab} + \vert w \vert_{bc}> (m+n)\vert w\vert_b \geq 2\vert w\vert_b$.
 Similarly, $\vert w'\vert_{ab}+ \vert w'\vert_{bc}< 2\vert w'\vert_b$ as $\beta'\geq 1$. It follows that
$\vert w \vert_{ab} + \vert w \vert_{bc}> \vert w'\vert_{ab}+ \vert w'\vert_{bc}$,
contradicting the fact that $w\equiv_M w'$. \qedSymC

Finally, it remains to deal with the situation when $\vert \print(\pi_{a,c}(v))\vert =2$, for all $v\in [w]$. Since $[w]$ is distinguishable, we may assume $\pi_{a,c}(w)= a^m c^n$ and $\pi_{a,c}(w')= c^na^m$, for some $w'\in [w]$ and positive integers $m$ and $n$. Since $[w]$ is not trivially distinguishable, it implies that $\vert w\vert_{abc}\geq 1$. However, clearly $\vert w'\vert_{abc}=0$, contradicting the fact that $w\equiv_M w'$.
\end{proof}

The following example shows that Conjecture~\ref{160720a} only stands in the case of ternary alphabets, i.e., Theorem~\ref{0507a}.

\begin{example}\label{exmp:conj}
Consider the ordered alphabet $\Sigma=\{a<b<c<d\}$ and the words $w=bcbab cbcdc bcbba bcbcc dccbb$ and $v=cbbab bcbcd ccbcb abcbc dcbcb$. We can check\footnote{e.g., using the freely available tool at www.github.com/LHutch1/Parikh-Matrices-Toolkit} that the corresponding $M$\!-equivalence class only contains these words, i.e., $[w]=\{w,v\}$, and
\[\Psi_\Sigma(w)=\begin{psmallmatrix}
	1 & 2 & 13 & 51 & 36\\
	0 & 1 & 11 & 61 & 51\\
	0 & 0 & 1 & 10 & 11\\
	0 & 0 & 0 & 1 & 2\\
	0 & 0 & 0 & 0 & 1
\end{psmallmatrix}.
\]

We note that $[w]$ is distinguishable in this case since $|w|_x \neq |v|_x$, for $x\in \{ac,bd,acd,abd\}$. However, clearly $[w]$ is not trivially distinguishable since all of its symbols on consecutive positions are also consecutive in the alphabet.\qedSymC
\end{example}
The above example can easily be extended to alphabets larger than $4$ letters.

\section{Minimal Hamming Distances for Ternary Alphabets}\label{sec:4}

In this section, we begin by providing the definition of minimal Hamming distances for $M$\!-classes introduced in~\cite{atanasiu2022erasure} and generalizing some elementary facts.

\begin{definition}
	Let $\Sigma$ be an alphabet and $w,w'\in\Sigma^*$ with $|w|=|w'|$. The \textit{Hamming distance} between $w$ and $w'$, denoted $d_H(w,w')$, is the number of positions in which the corresponding letters of $w$ and $w'$ differ.
\end{definition}

\begin{definition}\label{DefHamMClass}
Let $\Sigma=\{a_1<a_2<\dotsb<a_s\}$ and take $w\in\Sigma^*$ to be $M$\!-ambiguous.
The \textit{minimal Hamming distance} of $[w]$, denoted $d_H([w])$, is defined by:
	\begin{center}
		$d_H([w])=\min\{\,d_H(w',w'')\,|\,w',w''\in [w] \text{ and } w'\neq w''\,\}.$
	\end{center}
\end{definition}

The next two propositions generalize part of~\cite[Proposition~3.5]{atanasiu2022erasure} and~\cite[Lemma~3.7]{atanasiu2022erasure} to any ordered alphabet.  

\begin{proposition}\label{2711c}
Let $\Sigma=\{a_1<a_2<\dotsb<a_s\}$ and take $w\in\Sigma^*$ to be $M$\!-ambiguous. Then $[w]$ is trivially $\mathbb{P}$-distinguishable if and only if $d_H ([w])=2$. 
\end{proposition}

\begin{proof}
	The forward direction is trivial. Conversely, suppose $d_H([w])=2$. Then there exist $w', w''\in [w]$ such that $d_H(w',w'')=2$.
	We may assume $w'= w_1a_i w_2a_j w_3 $ and $ w''=w_1 a_j w_2 a_i w_3$, for some 
	$i<j$ and $w_1,w_2,w_3\in\Sigma^*$. 
	
	Observe that if $\vert w_2 a_j\vert_{a_{i+1}}\geq 1$, then 
	$\vert a_iw_2a_j\vert_{a_{i}a_{i+1}}>\vert a_jw_2a_i\vert_{a_ia_{i+1}}$
	and thus $\vert w'\vert_{a_{i}a_{i+1}}>\vert w''\vert_{a_ia_{i+1}}$,
	contradicting that $w'\equiv_M w''$. This implies that $j\geq  i+2$ and 
	$\vert w_2\vert_{a_{i+1}}=0$. It can then be deduced that $a_i w_2 a_j$ contains a factor of the form $a_{i'}a_{j'}$, for some $i'\leq i$ and $j'\geq i+2$. Hence, $[w]$ is trivially distinguishable. 
\end{proof}

\begin{proposition}\label{2711d}
	Let $\Sigma=\{a_1<a_2<\dotsb<a_s\}$ and take $w\in\Sigma^*$ to be $M$\!-ambiguous.
	The minimal Hamming distance of $[w]$ cannot be three. 	
\end{proposition}

\begin{proof}
	Assume $d_H([w])=3$. Then there exist $w', w''\in [w]$ with $d_H(w',w'')=3$. 
	As the argument in each case mirrors that of Proposition~\ref{2711c}, we consider only the case $w'= w_1a_i w_2a_j w_3 a_k w_4 $ and $ w''=w_1 a_k w_2 a_i w_3 a_j w_4$, for some 
	$i<j<k$ and $w_1,w_2,w_3, w_4\in\Sigma^*$. 
	
	Observe that if $\vert w_3 a_k\vert_{a_{j+1}}\geq 1$, then 
	$\vert a_i w_2a_j w_3 a_k\vert_{a_{j}a_{j+1}}>\vert a_k w_2 a_i w_3 a_j\vert_{a_ja_{j+1}}$
	and thus $\vert w'\vert_{a_{j}a_{j+1}}>\vert w''\vert_{a_j a_{j+1}}$,
	contradicting that $w'\equiv_M w''$. This implies that $k\geq  j+2$ and 
	$\vert w_3\vert_{a_{j+1}}=0$. It can then be deduced that $a_j w_3 a_k$ contains a factor of the form $a_{j'}a_{k'}$, for some $j'\leq j$ and $k'\geq j+2$. It follows that $[w]$ is trivially distinguishable and thus 
	$d_H([w])=2$, contradicting our assumption.	
\end{proof}





For the remainder of this section, we again work on the fixed ordered alphabet $\Sigma_3=\{a<b<c\}$. Consider the word $w=ab^{11}cbabcb^5$. It can be verified that
\[[w]=\{ \, b^i a b^{11-2i} c bab^{1+2i}c b^{5-i} \mid 0\leq i\leq 5 \,\}\]
with $d_H([w])=7$ (see~\cite[Example~5.1]{atanasiu2022erasure} and~\cite[Theorem~5.2]{atanasiu2022erasure}). Our second main result  asserts $7$ as one of a few (attainable) minimal Hamming distances of $M$\!-classes over $\Sigma_3$. Before that, we need a few lemmata.

\begin{lemma}\label{0907a}
Consider $w\in \Sigma_3^*$ with $\vert \print(\pi_{a,c} (w)  )\vert \in \{2,3\}$. If $[w]$ is not trivially $\mathbb{P}$-distinguishable, then $d_H([w])=4$ unless $w$ is $M$\!-unambiguous.
\end{lemma}

\begin{proof}
Suppose 	$[w]$ is not trivially distinguishable. 
First, we consider the case $\print(\pi_{a,c} (w)  ) = ac$. 
Then $w= uab^\beta cv$, for some $u\in \{a,b\}^*, v\in \{b,c \}^*$, and $\beta\geq 1$.
As $d_H([w])\notin \{2,3\}$ by Propositions~\ref{2711c} and~\ref{2711d}, it follows that if $uab^\beta$ is \mbox{$M$\!-ambiguous} over $\{a<b\}$ or $b^\beta c v$ is $M$\!-ambiguous over $\{b<c\}$, then, through the use of Rule $E2$, the distance between some two words in this $M$\!-class is $4$, hence $d_H([w])=4$.
Suppose $d_H([w])\neq 4$. Then $uab^\beta$ must be \mbox{$M$\!-unambiguous} and thus, by Theorem~\ref{1906bb}, it is of one of the following possibilities $b^\alpha a b^\beta, a^\alpha b^\beta$, or $a^{\alpha}bab^\beta$, for some $\alpha\geq 1$. Similarly, $b^\beta c v$ is either $b^\beta c b^\gamma, b^\beta c^\gamma$,  or $b^\beta cb c^\gamma$, for some $\gamma \geq 1$. Therefore, $w$ has one of the following nine forms: 
\[
\begin{gathered}
b^\alpha a b^\beta c b^\gamma,  b^\alpha a b^\beta c^{\gamma},  b^\alpha a b^\beta cbc^{\gamma}, a^\alpha b^\beta c b^\gamma,  a^{\alpha} b^\beta c^{\gamma}, \\ a^{\alpha}b^\beta cbc^{\gamma},a^{\alpha}bab^\beta cb^\gamma,  a^{\alpha}bab^\beta c^{\gamma},  a^{\alpha}bab^\beta cbc^{\gamma}.
\end{gathered} 
\]
However, all of these words are $M$\!-unambiguous over $\Sigma_3$ (see Remark~\ref{rem:M-unambiguity}).

Now, consider the case $\print(\pi_{a,c} (w)  ) = aca$. Then $w= uab^\beta cb^\gamma a v$, for some $u, v\in \{a,b\}^*$ and $\beta, \gamma\geq 1$. Suppose $d_H([w])\neq 4$.
Similarly, $uab^\beta$ must be \mbox{$M$\!-unambiguous} over $\{a<b\}$ and thus it is $b^\alpha a b^\beta, a^{\alpha} b^\beta$, or $a^{\alpha}bab^\beta$, for some $\alpha\geq 1$, and $b^\gamma av$ is 
	$b^\gamma a b^\delta, b^\gamma a^{\delta}$, or $b^\gamma ab a^{\delta}$, for some $\delta \geq 1$. Therefore, $w$ has one of the following nine forms: 
	\[
	\begin{gathered}
	b^\alpha a b^\beta c  b^\gamma a b^\delta, b^\alpha a b^\beta c b^\gamma a^{\delta}, b^\alpha a b^\beta cb^\gamma aba^{\delta},
	a^{\alpha}b^\beta c b^\gamma ab^\delta, a^{\alpha} b^\beta c b^\gamma a^{\delta},\\
	a^{\alpha} b^\beta cb^\gamma a b a^{\delta},a^{\alpha} bab^\beta cb^\gamma a b^\delta, a^{\alpha} bab^\beta cb^\gamma a^{\delta}, a^{\alpha}  bab^\beta cb^\gamma ab a^{\delta}.
	\end{gathered}
	\] 
	However, again, all of them are $M$\!-unambiguous over $\Sigma_3$ (see Remark~\ref{rem:M-unambiguity}).
	
 The cases where $\print(\pi_{a,c} (w)  )= ca$ or $cac$ are analogous. 
\end{proof}

\begin{lemma}\label{080823a}
Consider $w, w'\in \Sigma_3^*$ with $w\neq w'$ such that  $\pi_{a,c}(w)= \pi_{a,c}(w')\in \{acac, caca\}$. 
 If $w\equiv_M w'$ and the $M$\!-class $[w]$ is not trivially $\mathbb{P}$-distinguishable,  then $d_H(w,w')\in \{7,8\}$.	
\end{lemma}

\begin{proof}
We may assume $\pi_{a,c}(w)= \pi_{a,c}(w')= acac$ as the other case is analogous.
Suppose $w\equiv_M w'$ and $[w]$ is not trivially distinguishable.
Then $w= b^\alpha a b^\beta c  b^\gamma a b^\delta c b^\nu$ and $w'= b^{\alpha'} a b^{\beta'} c  b^{\gamma'} a b^{\delta'} c b^{\nu'}$, for some $\beta,\gamma,\delta, \beta', \gamma',\delta' \geq 1$ and $\alpha, \nu, \alpha', \nu'\geq 0$.
If $\alpha =\alpha'$, then 
letting $v= b^\beta c  b^\gamma a b^\delta c b^\nu$ and $v'= b^{\beta'} c  b^{\gamma'} a b^{\delta'} c b^{\nu'}$, by Remark~\ref{2912a} we see that $v\equiv_M v'$. However, since using Remark~\ref{rem:M-unambiguity} we know that $v$ is $M$\!-unambiguous, we conclude that $v=v'$, which contradicts in this case that $w\neq w'$. Hence, we may assume $\alpha < \alpha'$.

Since $\pi_{a,b}(w)
=b^\alpha a b^{\beta+\gamma} a b^{\delta+\nu}\equiv_M
b^{\alpha'} a b^{\beta'+\gamma'} a b^{\delta'+\nu'} = \pi_{a,b}(w')$, from the first part of Theorem~\ref{1906bb}
it follows that
$\alpha'-\alpha= (\delta'+\nu')-(\delta+\nu)$
and $\alpha+\beta +\gamma > \alpha'+\beta'+\gamma'$.
Similarly, from  $\pi_{b,c}(w)
=b^{\alpha +\beta} c b^{\gamma+\delta} cb^\nu\equiv_M
b^{\alpha' +\beta'} c b^{\gamma'+\delta'} cb^{\nu'} = \pi_{b,c}(w')$, we obtain
$(\alpha'+\beta') - (\alpha+\beta)= \nu'-\nu$.
Hence, $\beta'-\beta = \delta-\delta'$.
Now, since $\vert w\vert_{abc}= \vert w'\vert_{abc}$, we have
$2\beta +\gamma +2\delta= 2\beta' +\gamma' +2\delta'$, and it follows that $\gamma= \gamma'$. Furthermore, $\alpha+\nu = \alpha'+\nu'$ because  $\vert w\vert_b=\alpha+\beta +\gamma +\delta+\nu= \alpha'+\beta'+\gamma'+\delta'+\nu'=\vert w'\vert_b$.

Now, $\alpha < \alpha'$ implies that
$\nu >\nu' $. It follows that
$\alpha+\beta >\alpha'+\beta'$ and $\alpha+\beta +\gamma +\delta< \alpha'+\beta'+\gamma'+\delta'$. Having these two inequalities together with $\alpha<\alpha'$ and $\alpha+\beta +\gamma > \alpha'+\beta'+\gamma'$, it can now be verified that $d_H(w,w') =8$ unless 
$\alpha+\beta+2= \alpha'+\beta'+\gamma'+3$; that is, when the position of the first $c$ in $w$ coincides with that of the second $a$ in $w'$. In that exceptional case, $d_H(w,w')=7$. 	
\end{proof}

\begin{lemma}\label{0907b}
Let $w\in \Sigma_3^*$ be \mbox{$M$\!-ambiguous} with $\vert\print(\pi_{a,c} (w))\vert=4$. If the $M$\!-class $[w]$ is not trivially $\mathbb{P}$-distinguishable, then $d_H([w])\in \{4, 7,8\}$.
\end{lemma}

\begin{proof}
We may assume $\print(\pi_{a,c} (w))= acac$ as the other case is analogous. 
Suppose $[w]$ is not trivially distinguishable and $d_H([w])\neq 4$. Then 
$\pi_{a,c}(w)= a^mcac^n$, for some $m,n\geq 1$.
We need to show that $d_H([w])\in \{7,8\}$. 

Assume that $z$ and $z'$ are arbitrary distinct elements of $[w]$. By Theorem~\ref{0507a}, $[w]$ is indistinguishable and thus $\pi_{a,c}(z)\equiv_M \pi_{a,c} (w) \equiv_M \pi_{a,c}(z')$ over $\{a<c\}$, implying that $\pi_{a,c}(z)= \pi_{a,c}(z')= a^mcac^n$ as this word is $M$\!-unambiguous over $\{a<c\}$. 
 Then $z= uab^\beta cb^\gamma a b^\delta c v$  and $z'= u'ab^{\beta'} cb^{\gamma'} a b^{\delta'} c v'$,  for some $u,u'\in \{a,b\}^*$, $v,v'\in \{b,c\}^*$, and $\beta, \gamma, \delta, \beta', \gamma', \delta' \geq 1$.

Since $d_H([w])\notin \{2,3\}$ by Propositions~\ref{2711c} and~\ref{2711d}, it follows that both $uab^\beta$ and $u'ab^{\beta'}$ are \mbox{$M$\!-unambiguous} over $\{a<b\}$ since, otherwise, $d_H([w])= 4$ and this violates our supposition.
Hence, if $m\geq 2$, then $uab^\beta$ is $a^mb^{\beta}$ or $a^{m-1}bab^{\beta}$, while
 $u'ab^{\beta'}$ is $a^mb^{\beta'}$ or $a^{m-1}bab^{\beta'}$, and thus $a^{m-1}$ is a prefix  of both $z$ and $z'$. Similarly, if $n\geq 2$, then $b^\delta cv $ is  $b^{\delta}c^n$ or $b^\delta cbc^{n-1}$, while
 $b^{\delta'} cv' $ is  $b^{\delta'}c^n$ or $b^{\delta'} cbc^{n-1}$, and thus
 $c^{n-1}$ is a suffix of both $z$ and $z'$. 
  Following Remark~\ref{2912a}, canceling out the prefix $a^{m-1}$ and the suffix $c^{n-1}$ from both $z$ and $z'$ if applicable, we are left with distinct $\bar{z}$ and $\bar{z}'$ such that $\pi_{a,c}(\bar{z})=\pi_{a,c}(\bar{z}')= acac$, $\bar{z} \equiv_M \bar{z}'$, and $[\bar{z}]$ is not trivially distinguishable.  By Lemma~\ref{080823a},  $d_H(\bar{z},\bar{z}')\in \{7,8\}$. Since $d_H(z,z')=d_H(\bar{z},\bar{z}')$, it follows that $d_H([w])\in \{7,8\}$. 
\end{proof}

We are ready for our second main theorem, which was conjectured in~\cite{atanasiu2022erasure}.

\begin{theorem}\label{1911a}
	Let $w\in \Sigma_3^*$ be $M$\!-ambiguous. 
	Then $d_H([w])\in \{2,4,7,8\}$.	
\end{theorem}

\begin{proof}
 If $[w]$ is trivially distinguishable, then $d_H([w])=2$. 
Suppose $[w]$ is not trivially distinguishable. By Lemmas~\ref{1011a},~\ref{2906b}, and~\ref{1011b},  $\vert \print(\pi_{a,c} (w)) \vert \leq 4$. 
Clearly, $\vert\print (\pi_{a,c}(w))\vert\geq 1$ or else $w$ is \mbox{$M$\!-unambiguous}.
If 	$\print(\pi_{a,c} (w)) \in \{a,c\}$, then $w\in \{a,b\}^*$ or $w\in \{b,c\}^*$. In either case, $d_H([w])=4$.
The cases where $\vert \print(\pi_{a,c} (w)) \vert \in \{2,3,4\}$ are addressed by Lemmas~\ref{0907a} and~\ref{0907b}. 
\end{proof}

\section{Conclusion}

It was conjectured in~\cite{serbanuta2006injectivity}, that for any ordered alphabet $\Sigma$ and any $u,v\in\Sigma^*$ and $a\in\Sigma$, if $uav$ is $M$\!-ambiguous, then $uaav$ is also $M$\!-ambiguous. This is true for the ternary alphabet, but counterexamples were later found for the quaternary alphabet~\cite{teh2018strongly}. Hence, our work here presents yet another evidence where some promising conjectures in Parikh matrices may turn out to be true for the ternary alphabet but false for the quaternary alphabet. Interested readers are referred to~\cite{poovanandran2020ambiguity} for a further study on the possible alterations of the \mbox{$M$\!-ambiguity} property by repeated duplication of letters in a word.

Regarding our contribution, we showed that Conjecture~\ref{160720a} is false for ordered alphabets with more than $3$ letters. Now, due to Theorem~\ref{0507a}, we get the following characterization:

\begin{theorem}
Let $w\in \Sigma_3^*$ be $M$\!-ambiguous. The $M$\!-equivalence class $[w]$ is \mbox{$\mathbb{P}$-distinguishable} if and only if $[w]$ is trivially $\mathbb{P}$-distinguishable.	
\end{theorem}

In Table~\ref{151017b} we give more $4$-letter alphabet's counterexamples, and include, for each of their $M$\!-classes, the corresponding minimal Hamming distances. 

\begin{table}[h]
	\begin{center}
		\begin{tabular}{ | p{22em} | p{6em} | }
			\hline
			$M$-class & \tiny{minimal Hamming distance}\\ \hline
			$\{bcbabccdcbbabbccdccb, cbbabbccdccbabbcdcbc \} $& 12       \\ \hline
			$\{bcbabcbcdcbbabbccdccb, cbbabbcbcdccbabbcdcbc  \}$ & 14 \\ \hline
			$\{ bcbabcbcdcbbabcbccdccb, cbbabbcbcdccbabcbcdcbc \}$ & 16\\ \hline
					\end{tabular}
		\vspace{1mm}
		\caption{Counterexamples to Conjecture~\ref{160720a}}\label{151017b}
	\end{center}
\end{table}
 
These counterexamples further show that $M$-classes for the quaternary alphabet exhibit wilder behaviour compared to the ternary alphabet when it comes to minimal Hamming distances. This tendency somehow strengthens the connection in between minimal Hamming distance and the idea of $\mathbb{P}$-distinguishability, when considering our main results.
We leave as a future direction the study of possible minimal Hamming distances for the quaternary or higher alphabets. 
Here, we simply pose the following related question to which an unexpected negative answer would be beyond our expectations.

\begin{question}
	Is there a fixed upper bound on the minimal Hamming distance of any $M$\!-class for any ordered alphabet? 
\end{question} 
 
Finally, we point out that further extensions of the downward projections from~\cite{dick2021reducing} could follow the word projections model proposed in~\cite{serbanuta2004extending}. Such projections could take the form of {\em upward projections}, in which case further symbols are inserted in the corresponding words, or {\em word projections} via well chosen sequences. Such an approach could lead to a further ambiguity reduction, in the larger sense (when considered along Parikh matrices, and not only in terms of $\mathbb{P}$-distinguishability). 

Furthermore, such new projections could also express further connections to the idea of minimal Hamming distance and shed more light on the characteristics that these exhibit. We suspect that a notion, which is ``often mentioned but not much investigated in the literature" \cite[Page 196]{salomaa2005connections},  that might need further attention in the case of word projections is that of $t$-spectrum, which might favour our analysis of the transformation. 

\end{document}